\documentclass[a4paper,12pt]{article}
\usepackage{}
\usepackage{mathrsfs}
\usepackage{amsfonts}
\usepackage{amscd,color}
\usepackage{amsmath,amsfonts,amssymb,amscd}
\usepackage{indentfirst,graphicx,epsfig}
\usepackage{graphicx,psfrag}
\input{epsf}

\newtheorem{df}{Definition}[section]
\newtheorem{thm}{Theorem}[section]

\newtheorem{conjecture}{Conjecture}[section]
\newtheorem{lem}{Lemma}[section]

\newtheorem{cor}{Corollary}[section]

\newenvironment {proof} {\noindent{\em Proof.}}{\hspace*{\fill}$\Box$\par\vspace{4mm}}

\title{Solution to a conjecture on\\ the proper connection number of graphs\footnote{Supported by NSFC No.11371205, 11531011, and ``973" program No.2013CB834204.}}

\author{\small {Fei Huang, Xueliang Li, Zhongmei Qin}\\
{\small  Center for Combinatorics and LPMC}\\
{\small Nankai University, Tianjin 300071, P.R. China}\\
{\small Email: huangfei06@126.com, lxl@nankai.edu.cn, qinzhongmei90@163.com}\\
{\small Colton Magnant}\\
{\small Department of Mathematical Sciences}\\
{\small Georgia Southern University, Statesboro, GA 30460-8093, USA}\\
{\small Email: cmagnant@georgiasouthern.edu}
}
\date{}
\begin{document}
\maketitle
\begin{abstract}
 A path in an edge-colored graph is called a proper path if no two adjacent edges of the path receive the same color. For a connected graph $G$, the proper connection number $pc(G)$ of $G$ is defined as the minimum number of colors needed to color its edges, so that every pair of distinct vertices of $G$ is connected by at least one proper path in $G$. Recently, Li and Magnant in [Theory Appl. Graphs 0(1)(2015), Art.2] posed the following conjecture: If $G$ is a connected noncomplete graph of order $n \geq 5$ and minimum degree $\delta(G) \geq n/4$, then $pc(G)=2$. In this paper, we show that this conjecture is true except for two small graphs on 7 and 8 vertices, respectively. As a byproduct we obtain that if $G$ is a connected bipartite graph of order $n\geq 4$ with $\delta(G)\geq \frac{n+6}{8}$, then $pc(G)=2$. \\[2mm]
\textbf{Keywords:} proper connection number; proper-path coloring; bridge-block tree.\\
\textbf{AMS subject classification 2010:} 05C15, 05C40, 05C07.\\
\end{abstract}

\section{Introduction}

All graphs in this paper are undirected, finite and simple. We follow \cite{BM} for graph theoretical notation and terminology not defined here. Let $G$ be a connected graph with vertex set $V(G)$ and edge set $E(G)$. For any two disjoint subsets $X$ and $Y$ of $V(G)$, we use $E_G(X,Y)$ to denote the set of edges of $G$ that have one end in $X$ and the other in $Y$. Denote by $|E_G(X,Y)|$ the number of edges in $E_G(X,Y)$. If $Y=V(G)\setminus X$, we use $d(X)=|E_G(X,Y)|$ for short. A graph $G$ is called {\it Hamilton-connected} if there is a Hamilton (spanning) path between any pair of vertices in $G$.

Let $G$ be a nontrivial connected graph with an associated {\it edge-coloring} $c : E(G)\rightarrow \{1, 2, \ldots, t\}$, $t \in \mathbb{N}$, where adjacent edges may have the same color. If adjacent edges of $G$ are assigned different colors by $c$, then $c$ is a {\it proper (edge-)coloring}. For a graph $G$, the minimum number of colors needed in a proper coloring of $G$ is referred to as the {\it edge-chromatic number} of $G$ and denoted by $\chi'(G)$. A path of an edge-colored graph $G$ is said to be a {\it rainbow path} if no two edges on the path have the same color. The graph $G$ is called {\it rainbow connected} if for any two vertices there is a rainbow path of $G$ connecting them. An edge-coloring of a connected graph is a {\it rainbow connecting coloring} if it makes the graph rainbow connected. This concept of rainbow connection of graphs was
introduced by Chartrand et al.~\cite{CJMZ} in 2008. The \emph{rainbow connection number} $rc(G)$ of a
connected graph $G$ is the smallest number of colors that are needed in order to make $G$ rainbow connected. The readers who are interested in this topic can see \cite{LSS,LS} for a survey.

Motivated by rainbow coloring and proper coloring in graphs, Andrews et al.~\cite{ALLZ} and Borozan et al.~\cite{BFGMMMT} introduced the concept of proper-path coloring. Let $G$ be a nontrivial connected graph with an edge-coloring. A path in $G$ is called a \emph{proper path} if no two adjacent edges of the path are colored with the same color. An edge-coloring of a connected graph $G$ is a \emph{proper-path coloring} if every pair of distinct vertices of $G$ are connected by a proper path in $G$. If $k$ colors are used on the edges of $G$, then $c$ is referred to as a {\it proper-path $k$-coloring}. An edge-colored graph $G$ is {\it proper connected} if any two vertices of $G$ are connected by a proper path. For a connected graph $G$, the {\it proper connection number} of $G$, denoted by $pc(G)$, is defined as the smallest number of colors that are needed in order to make $G$ proper connected.

The proper connection of graphs has the following application background. When building a communication network between wireless signal towers, one fundamental requirement is that the network be connected. If there cannot be a direct connection between two towers $A$ and $B$, say for example if there is a mountain in between, there must be a route through other towers to get from $A$ to $B$. As a wireless transmission passes through a signal tower, to avoid interference, it would help if the incoming signal and the outgoing signal do not share the same frequency. Suppose we assign a vertex to each signal tower, an edge between two vertices if the corresponding signal towers are directly connected by a signal and assign a color to each edge corresponding to the assigned frequency used for the communication. Then the total number of frequencies needed to assign frequencies to the connections between towers so that there is always a path avoiding interference between each pair of towers is precisely the proper connection number of the corresponding graph.

Let $G$ be a nontrivial connected graph of order $n$ (number of vertices) and size $m$ (number of edges). Then the proper connection number of $G$ satisfies the following natural bounds:
$$1\leq pc(G)\leq \min\{ rc(G), \chi'(G)\}\leq m.$$
Furthermore, $pc(G) = 1$ if and only if $G = K_n$, and $pc(G) = m$ if and only if $G = K_{1,m}$ is a star of size $m$. There have been a lot of results on the proper connection of graphs. For more details, we refer to \cite{ALLZ,BFGMMMT,GLQ,LWY} and a dynamic survey \cite{LC}.

When density is provided by a minimum degree assumption, the following result is sharp.

\begin{thm}\cite{BFGMMMT}\label{thm1}
If $G$ is a connected noncomplete graph with $n \geq 68$ vertices and $\delta(G) \geq n/4$, then $pc(G)=2$.
\end{thm}

Although the bound on the minimum degree in Theorem \ref{thm1} is best possible, there is nothing to suggest that the bound $n \geq 68$ is required. Very recently, Li and Magnant \cite{LC} posed the following conjecture in the dynamic survey \cite{LC}.
\begin{conjecture}(\cite{LC}, Conjecture 5.1)\label{conj1}
If $G$ is a connected noncomplete graph with $n \geq 5$ vertices and $\delta(G) \geq n/4$, then $pc(G)=2$.
\end{conjecture}
They thought that the bound $5$ on $n$ would be the best possible since $pc(K_{1,3})=3$.

We are interested in this conjecture and study the proof of Theorem~\ref{thm1} carefully. Unfortunately, we find that their proof has some gaps. For example, in \cite{BFGMMMT} Page 2558, Line 38, there is a conclusion  ``it is easy to check that $\langle \{v\} \cup C_1 \rangle$ contains a Hamilton path $P_1$ such that $v \in endpoints(P_1)$." In fact, we know that the minimum degree of $C_1$ is at least $n/4-1$, whereas $|C_1| \leq \frac{n-1}{2}$. These two conditions can only guarantee that $C_1$ has a Hamilton path rather than a Hamilton cycle. So it is not easy to show the conclusion. Another place is on Page 2558, Line 44, that ``it is easy to check that there are exactly two components $C_{21}, C_{22}$ with $|C_{21}| \leq |C_{22}|$ in $C_2-S$." We find that the case three components may exist. One more place is on Page 2559, Line 12, that ``By Theorem 9, $C_3$ is Hamiltonian so it contains a spanning path $P$ with $v \in endpoints(P)$." Also, it is not very easy to show that $C_3$ is Hamiltonian by their Theorem 9.

Because of the above gaps, we have to find new methods to attack Conjecture~\ref{conj1}. As a result, we completely solve Conjecture~\ref{conj1} and find two small counterexample graphs $G_1, G_2$ shown in Figure~1, for which it is easy to check that $pc(G_1)=pc(G_2)=3$.

\begin{figure}
  \centering
 \scalebox{1}{\includegraphics{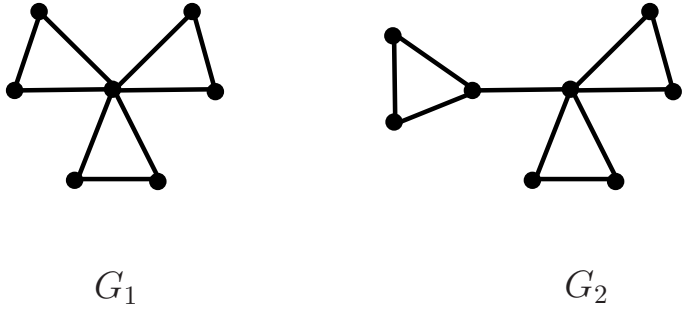}}\\
  \caption{Two counterexamples of Conjecture~\ref{conj1}}
\end{figure}

Hence, we modify Conjecture \ref{conj1} and give a new statement as follows.

\begin{thm}\label{th1}
Let $G$ be a connected noncomplete graph of order $n \geq 5$. If $G \notin \{G_1, G_2\}$ and $\delta(G) \geq  n/4$, then $pc(G)=2$.
\end{thm}

\section{Preliminaries}

At the beginning of this section, we list some known fundamental results and definitions on proper-path
coloring.

\begin{lem}\cite{ALLZ}\label{lem1}
If $G$ is a nontrivial connected graph and $H$ is a connected spanning
subgraph of $G$, then $pc(G) \leq pc(H)$. In particular, $pc(G) \leq pc(T)$
for every spanning tree $T$ of $G$.
\end{lem}

\begin{lem}\cite{ALLZ}\label{lem2}
If $T$ is a nontrivial tree, then $pc(T) = \chi'(T) = \Delta(T)$.
\end{lem}

\begin{lem}\cite{ALLZ}\label{lem3}
Let $G$ be a connected graph and $v$ a vertex not in $G$. If $pc(G)=2$, then $pc(G \cup v)=2$ as long as $d(v)\geq 2$, that is, there are at least two edges connecting $v$ to $G$.
\end{lem}

Given a colored path $P = v_1v_2 \ldots v_{s-1}v_s$ between any two vertices $v_1$ and $v_s$, we denote by $start(P)$ the color of the first edge in the path, i.e., $c(v_1v_2)$, and by $end(P)$ the color of the last edge, i.e., $c(v_{s-1}v_s)$. If $P$ is just the edge $v_1v_s$, then $start(P) = end(P) = c(v_1v_s)$.

\begin{df} \cite{BFGMMMT}\label{df1}
Let $c$ be an edge-coloring of $G$ that makes $G$ proper connected. We say
that $G$ has the \emph{strong property} under $c$ if for any pair of vertices $u, v \in V(G)$, there exist two proper paths $P_1, P_2$ from $u$ to $v$ (not necessarily disjoint) such that $start(P_1) \neq start(P_2)$ and $end(P_1) \neq end(P_2)$.
\end{df}

In \cite{BFGMMMT}, Borozan et al.~studied the proper-connection number in 2-connected graphs. Also, they presented a result which improves upon the upper bound $pc(G) \leq \Delta(G)+1$ to the best possible whenever the graph $G$ is bipartite and 2-connected.

\begin{lem}\cite{BFGMMMT}\label{lem4}
Let $G$ be a graph. If $G$ is bipartite and 2-connected, then $pc(G)= 2$ and
there exists a 2-edge-coloring $c$ of $G$ such that $G$ has the strong property under~$c$.
\end{lem}

\begin{cor}\cite{BFGMMMT}\label{cor1}
Let $G$ be a graph. If $G$ is 3-connected and noncomplete, then $pc(G)= 2$ and
there exists a 2-edge-coloring $c$ of $G$ such that $G$ has the strong property under $c$.
\end{cor}

\begin{lem}\cite{BFGMMMT}\label{lem5}
Let $G$ be a graph. If $G$ is 2-connected, then $pc(G) \le 3$ and there exists
a 3-edge-coloring $c$ of $G$ such that $G$ has the strong property under $c$.
\end{lem}

\begin{lem}\cite{HLW1}\label{lem6}
Let $H = G \cup \{v_1\} \cup \{v_2\}$. If there is a proper-path $k$-coloring $c$ of $G$
such that $G$ has the strong property under $c$, then $pc(H) \leq k$ as long as $v_1, v_2$ are not isolated vertices of $H$.
\end{lem}

In \cite{BFGMMMT}, they also studied the proper connection number for 2-edge-connected graphs by (closed) ear-decomposition. But they did not give the details of the proof. In \cite{HLW}, Huang et al.~gave a rigorous  proof for the result on 2-edge-connected graphs depending on Lemmas \ref{lem4} and \ref{lem5}.

\begin{lem} \cite{HLW}\label{lem7}
If $G$ is a connected bridgeless graph with $n$ vertices, then $pc(G) \leq 3$.
Furthermore, there exists a 3-edge-coloring $c$ of $G$ such that $G$ has the strong property
under $c$.
\end{lem}

\begin{lem} \cite{BFGMMMT,HLW}\label{lem8}
If $G$ is a bipartite connected bridgeless graph with $n$ vertices, then $pc(G) \leq
2$. Furthermore, there exists a 2-edge-coloring $c$ of $G$ such that $G$ has the strong property under $c$.
\end{lem}

They also stated a very useful result for graphs with cut-edges.

\begin{lem}\cite{HLW}\label{lem9}
Let $G$ be a graph with a cut-edge $v_1v_2$, and $G_i$ be the connected graph
obtained from $G$ by contracting the connected component containing $v_i$ of $G-v_1v_2$ to a
vertex $v_i$, where $i = 1, 2$. Then $pc(G) = \max\{pc(G_1), pc(G_2)\}$.
\end{lem}

\section{Proof of Theorem \ref{th1}}

At the beginning of this section, we list some useful results as follows.

\begin{thm}\cite{BM}\label{thm2}
Every loopless graph $G$ contains a bipartite spanning subgraph $H$ such that $d_H(v) \geq \frac{1}{2}d_G(v)$ for all $v \in V$.
\end{thm}

In fact, the process of the proof of Theorem \ref{thm2} implies the following stronger result.
\begin{thm}\cite{BM}\label{thm3}
Let $G$ be a loopless graph. Then each bipartite spanning subgraph $H$ of $G$ with the greatest possible number of edges satisfies $d_H(v) \geq \frac{1}{2}d_G(v)$ for all $v \in V$.
\end{thm}

\begin{thm}\cite{Wi}\label{thm4}
Let $G$ be a graph with $n$ vertices. If $\delta(G)\geq \frac{n-1}{2}$, then $G$ has a Hamilton path. Moreover, if $\delta(G)\geq \frac{n}{2}$, then $G$ has a Hamilton cycle. Also, if $\delta(G)\geq \frac{n+1}{2}$, then $G$ is Hamilton-connected.
\end{thm}

\begin{thm}\cite{Wi}\label{thm5}
Let $G$ be a graph with $n$ vertices. If $\delta(G)\ge \frac{n+2}{2}$, then $G$ is panconnected meaning that, between any pair of vertices in $G$, there is a path of every length from 2 to $n-1$.
\end{thm}

Also we use the following easy fact as a matter of course.

\noindent\textbf{Fact 1}. Every 2-connected graph $G$ is either Hamiltonian, or contains a cycle $C$ with at least $2\delta(G)$ vertices.\\

Let $B\subseteq E$ be the set of cut-edges of a graph $G$. Let $\mathcal{C}$ denote the
set of connected components of $G'= (V;E\setminus B)$. There are two types of elements in $\mathcal{C}$,
singletons and connected bridgeless subgraphs of $G$. Contracting each element of $\mathcal{C}$ to a vertex, we obtain a new graph $G^*$. It is easy to see that $G^*$ is the well-known \emph{bridge-block tree} of $G$, and the edge set of $G^*$ is $B$. An element of $\mathcal{C}$ which corresponds to a leaf in $G^*$ is called an \emph{end-block} of $G$.

Our first main result proves Theorem~\ref{th1} when $n \geq 9$.

\begin{thm}\label{thm6}
Let $G$ be a connected noncomplete graph of order $n\geq 9$. If $\delta(G)\geq n/4$, then $pc(G)=2$.
\end{thm}

\begin{proof}
If $G$ contains a bridgeless bipartite spanning subgraph, then $pc(G)=2$ by Lemmas \ref{lem1} and \ref{lem8}. Next, we assume that every bipartite spanning subgraph of $G$ has a cut-edge. Let $H$ be a bipartite spanning subgraph of $G$ with the greatest possible number of edges. From Theorem \ref{thm3}, we know that $\delta_H(v)\geq \frac{1}{2}\delta_G(v)\geq n/8$, which implies that $\delta(H)\geq 2$. Hence, we have that each end-block of $H$ is a maximal connected bridgeless bipartite subgraph, and so it contains at least $4$ vertices. In order to guarantee the minimum degree of $H$, we know that each end-block of $H$ contains at least $n/4$ vertices. Hence, $|V(L)|\geq \max\{4, n/4\}$ for each end-block $L$ of $H$. Let $e=x_1x_2$ be an arbitrary cut-edge of $H$, it follows that $H-e$ has two components, denoted by $I_1$ and $I_2$. We present a claim here.\\

{\it Claim 1.}  $|E_G(I_1, I_2)|\le 2.$

Note that $I_1$ and $I_2$ are bipartite, each with at least 4 vertices. Let $(U_i, V_i)$ be the bipartition of $I_i$ such that $x_1\in U_1$ and $x_2\in U_2$. If there is an edge $e_1\in (E_G(U_1,U_2)\cup E_G(V_1, V_2))\setminus \{e\}$, or if there are two edges $e_2,e_3\in E_G(U_1,V_2)\cup E_G(V_1, U_2)$, we can find a bipartite subgraph $H'=H+e_1$ or $H''=H-e+e_2+e_3$ of $G$ which has larger size than $H$, a contradiction. Hence, we know that $E_G(U_1,U_2)\cup E_G(V_1, V_2)=\{e\}$ and $|E_G(U_1,V_2)\cup E_G(V_1, U_2)|\leq 1$, or in other words, $|E_G(I_1, I_2)|\leq 2$, and so Claim~1 holds.\\

We choose a bipartite spanning subgraph, still denoted by $H$, of $G$ satisfying the following two conditions:

(1) $H$ has the maximum number of edges, and subject to condition (1),

(2) $\Delta(H^*)$ is as small as possible, where  $H^*$ is the bridge-block tree of $H$.

From Lemmas \ref{lem6}, \ref{lem8} and \ref{lem9}, we see that the result holds if $\Delta(H^*)\leq 2$. Hence, we only need to consider $\Delta(H^*)\geq 3$ in the following. On the other hand, we can also get that $H$ has at most three end-blocks since each end-block of $H$  contains at least $n/4$ vertices. This implies that $\Delta(H^*)\leq 3$. Hence, we have $\Delta(H^*)= 3$, and there is only one vertex attaining the maximum degree of $H^*$. Let $b_0\in V(H^*)$ be the vertex with degree $3$ and $B_0$ be the subgraph of $H$ corresponding to $b_{0}$. Let $D_1, D_2, D_3$ be the connected components of $H-V(B_0)$ and $L_i$ be the end-block included in $D_i$ for $i=1,2,3$. Since $|L_i|\ge 4$ for $i=1,2,3$, it follows that $n\ge |B_0|+|L_1|+|L_2|+|L_3|\ge 1+3\times 4=13$. So, in the following, we always assume $n\ge 13$. From Claim 1, we know that $|E_G(B_0, D_i)|\leq 2$ for $1\leq i\leq 3$. Hence, $d(b_0)\leq 6$. To complete our proof, we need the following claims. \\

{\it Claim 2.} $H^*\cong K_{1,3}$, and $|L_i|=|D_i|\ge n/4+1$ for $i=1,2,3$.

From Claim 1, we have that $d(L_i)\le 2$. In order to guarantee the minimum degree condition $\delta(G)\ge n/4$, we can get that $|L_i|\ge n/4+1$ for $i=1,2,3$. If $13\leq n\leq 16$, we know $|L_i|\ge 5$ for $i=1,2,3$. It follows that $|L_i|=|D_i|=5 \ (i=1,2,3)$ and $|B_0|=1$. Thus, Claim~2 holds. We may assume $n\ge 17$. Then, we know that $\delta(H)\geq 3$. Suppose $H^*\not\cong K_{1,3}$. Take a vertex $v$ with degree $2$ in $H^*$. Denote by $B(v)$ the block of $H$ corresponding to $v$.  Therefore, each vertex of $H^*$ other than $b_0$ corresponds to a maximal connected bridgeless subgraph of $H$ since $\delta(H)\geq 3$. By employing the minimum degree condition of $H$, we have that $B(v)$ contains at least $n/4$ vertices. Hence, $|V(H)|\geq |B_0|+|L_1|+|L_2|+|L_3|+|B(v)|\geq n+1$, a contradiction. Thus, we can also get that $H^*\cong K_{1,3}$. \\

{\it Claim 3.} $E_G(L_i, L_j)=\emptyset$ for $1\leq i\neq j\leq 3$.

Suppose, to the contrary of the claim, we assume without loss of generality that $E_G(L_1, L_2)\neq \emptyset$. Let $e_1\in E_G(L_1, L_2)$, and let $e_2$ be the cut-edge incident with $L_2$ in $H$. Let $H_1=H-e_2+e_1$. Note that $H_1$ is also a maximum bipartite spanning subgraph. Since $D_2$ is bridgeless from Claim 2, it follows that $\Delta(H_1^*)=\Delta(H^*)-1$, which contradicts the choice of $H$.\\

To prove our result, we distinguish the following three cases based on the value of $n$. Recall that we may already assume $n \geq 13$.

\textbf{Case 1.} $n\geq 25.$

In this case, we know that $\delta(G)\geq 7$. From Claim~2, we know that $D_i \ (1\leq i\leq 3)$ is bridgeless, and $|D_i|\geq n/4+1$ for $i=1,2,3$. So, we have $|B_0|\leq n/4-3$. Since $d(b_0)\leq 6$, the minimum degree condition $\delta(G)\geq 7$ cannot be satisfied when $B_0$ is a singleton. If $B_0$ is not a singleton, then there exists a vertex $v$ of $B_0$ such that $|E_G(\{v\}, V(G)\setminus B_0)| \leq 3$. Since $|B_0|\leq n/4-3$, it follows that $d_G(v) \leq n/4-4+3< n/4$, contradicting the assumption that $\delta(G) \geq n/4$. This implies that this case cannot occur since $\Delta(H^*)=3$.

\textbf{Case 2:} $17\leq n\leq 24.$

By Claims~2 and~3, we again get that $D_i \ (1\leq i\leq 3)$ is bridgeless, and $E_G(D_i, D_j)=\emptyset$ for $1\leq i\neq j \leq 3$. Moreover, we have that $|D_i|\geq n/4+1$ for $i=1,2,3$, and $|B_0|\leq n/4-3$.
Since $17\leq n\leq 24$, we have that $|D_i|\geq 6$ for $i=1,2,3$, and $|B_0|\leq 3$.  It follows that $B_0$ is a singleton, and $n\geq 3\times 6+1=19$. Let $B_0=\{v\}$.

\textbf{Subcase 2.1:} $19\leq n\leq 20.$

 In this case, we know that $\delta(G)\geq 5$, and there are only two possibilities: (1) $|D_1|=|D_2|=|D_3|=6$; and (2) $|D_1|=|D_2|=6$ and $|D_3|=7$.
 Utilizing Theorems \ref{thm4} and \ref{thm5} and the minimum degree condition $\delta(G)\geq 5$, it is easy to check that $D_1$ and $D_2$ are panconnected, and $D_3$ is Hamilton-connected. Since $d_G(v) \geq 5$ and $|E_G(B_0, D_i)| \leq 2$, we have that at least one of $D_1$ and $D_2$ satisfies that $|E_G(B_0,D_i)|=2$, say $D_1$. Then, we can find a cycle $C$ of length 6 containing $v$ in $G[\{v\}\cup D_1]$. We see that $G[\{v\}\cup D_i]$ contains a Hamilton path $P_i$ such that $v \in endpoints(P_i)$ for $i=2,3$. Let $S=C \cup P_2 \cup P_3$. It is easy to check that $pc(S)=2$. Note that $G$ contains $S$ as a subgraph and $V(G)\setminus V(S)$ has only one vertex with degree at least 5 in $G$. Then, $pc(G)=2$ by Lemma \ref{lem3}.

\textbf{Subcase 2.2:} $21\leq n\leq 24.$

In this case, we have that $\delta(G)\geq 6$, so $|D_i|\geq 7$. Hence, we know that $n\geq 3\times 7+1=22$.  There are four cases: (1) $|D_1|=|D_2|=|D_3|=7$; (2) $|D_1|=|D_2|=7$ and $|D_3|=8$; (3) $|D_1|=7$ and $|D_2|=|D_3|=8$; (4) $|D_1|=|D_2|=7$ and $|D_3|=9$.
Utilizing Theorems~\ref{thm4} and~\ref{thm5} and the minimum degree condition $\delta(G)\geq 5$, it is easy to check that in the first three cases, $D_i$ is panconnected for $1\leq i\leq 3$; in the forth case, $D_1$ and $D_2$ are panconnected, and $D_3$ is Hamilton-connected. The conclusion $pc(G)=2$ can be easily checked by a similar substructure argument as in Subcase 2.1.

\textbf{Case 3:} $13\leq n\leq 16.$

From Claim 1, we know that  $|L_i|=|D_i|=5 \ (i=1,2,3)$  and $|B_0|=1$. Let $B_0=\{v\}$. Utilizing Theorem~\ref{thm4} and the minimum degree condition $\delta(G)\geq 4$, it is easy to check that $D_i$ is Hamilton-connected for $1\leq i\leq 3$.  Since $d_G(v) \geq 4$ and $|E_G(B_0, D_i)| \leq 2$, we have that at least one of $D_i \ (i=1,2,3)$  satisfies that $|E_G(B_0,D_i)|=2$, say $D_1$. Then, we can find a Hamilton cycle $C$ in $G[\{v\}\cup D_1]$. We see that $G[\{v\}\cup D_i]$ contains a Hamilton path $P_i$ such that $v \in endpoints(P_i)$ for $i=2,3$. Let $S=C \cup P_2 \cup P_3$. It is easy to check that $pc(S)=2$. Then, we can get that $pc(G)\le pc(S) =2$ by Lemma \ref{lem1}.

The proof of Theorem~\ref{thm6} is now complete.
\end{proof}

Finally, we consider the case when $n$ is small.

\begin{thm}\label{thm7}
Let $G$ be a connected noncomplete graph with $5 \leq n \leq 8$ vertices. If $G \notin \{G_1, G_2\}$ and $\delta(G) \geq 2$, then $pc(G)=2$.
\end{thm}

\begin{proof}
It is easy to check that $pc(G)=2$ for $5 \leq n \leq 6$. Now let $n \geq 7$. Denote by $\kappa(G)$ the connectivity of $G$ and $C(G)$ a longest cycle of $G$. If $\kappa(G) \geq 3$, then by Corollary \ref{cor1}, we have $pc(G)=2$. So, we may assume that $\kappa(G)=1$ or $2$.

For $n=7$, we divide the proof into two cases according to the value of $\kappa(G)$.

\textbf{Case 1:} $\kappa(G)=1$.

Let $v$ be a cut-vertex of $G$ and let $C_1, \ldots, C_{\ell}$ be the components of $G \backslash v$ such that $|C_1| \leq \ldots \leq |C_{\ell}|$. By the minimum degree condition, we see that $|C_1| \geq 2$, and so $\ell =2$ or $3$. If $\ell =3$, then $G =G_1$, which contradicts the assumption. Thus, $\ell =2$. Since $n=7$, we further divide the proof into two subcases:

\textbf{Subcase 1.1:} $|C_1|=2$, $|C_2|=4$.

Let $V(C_1)=\{u_1, u_2\}$, $V(C_2)=\{w_1, w_2, w_3, w_4\}$. Since $\delta(G)\geq 2$, we have that $G[\{v\} \cup C_1]$ is a $K_3$. Since $G$ is connected, there is at least one edge from $v$ to $C_2$. Without loss of generality, we assume $vw_1 \in E(G)$. Since $d_G(w_1)\geq 2$, we know that $w_1$ have at least one neighbor in $\{w_2,w_3, w_4\}$. Assume that $w_1w_2\in E(G)$. Then there is a path $P=u_2u_1vw_1w_2$. We have $pc(P)=2$. If $|E(w_i, P)|\geq 2$ for $i=3,4$, then by Lemma \ref{lem3}, $pc(G)=2$. Thus, we may assume that $|E(w_3, P)| \leq 1$. By the minimum degree condition, $w_3w_4 \in E(G)$. If $w_iw_2 \in E(G)$ for $i=3$ or $4$, then there is a Hamilton path $u_2u_1vw_1w_2w_3w_4$ or $u_2u_1vw_1w_2w_4w_3$, and thus $pc(G)=2$. Hence, we assume that $w_iw_2 \notin E(G)$ for $i=3,4$. Therefore, $w_iw_1 \in E(G)$ for $i=3$ or $4$ since $C_2$ is connected. Since $d_G(w_2) \geq 2$, we have $w_2v \in E(G)$. We also can find a Hamilton path $u_2u_1vw_2w_1w_3w_4$ or $u_2u_1vw_2w_1w_4w_3$, and thus $pc(G)=2$.

\textbf{Subcase 1.2:} $|C_1|=|C_2|=3$.

Let $V(C_1)=\{u_1, u_2, u_3\}$, $V(C_2)=\{w_1, w_2, w_3\}$. Since $C_1, C_2$ are connected, we may assume that $C_1$ contains a path $P_1=u_1u_2u_3$ and $C_2$ contains a path $P_2=w_1w_2w_3$. If $vu_1 \in E(G)$ or $vu_3 \in E(G)$, then we can find a Hamilton path $P_1'$ in $G[\{v\}\cup C_1]$ such that $v \in endpoints(P_1')$. Otherwise, we have $vu_2 \in E(G)$. Since $d_G(u_1)\geq 2$ and $d_G(u_3) \geq 2$, we have $u_1u_3 \in E(G)$. We also can find a Hamilton path $P_1'$ in $G[\{v\}\cup C_1]$ such that $v \in endpoints(P_1')$. Similarly, we can find a Hamilton path $P_2'$ in $G[\{v \}\cup C_2]$ such that $v \in endpoints(P_2')$. Then $P_1' \cup P_2'$ is a Hamilton path of $G$, and hence, $pc(G)=2$.

\textbf{Case 2:} $\kappa(G)=2$.

By Fact 1, we have $|C(G)|\geq 4$.

If $|C(G)|=4$, then $C(G)$ has an ear since $G$ is 2-connected. Let $P_1$ be a longest ear of $C(G)$. It is easy to check that $P_1$ has length 2 and the end-vertices of $P_1$ are the antipodal vertices of $C(G)$. Then,  $G$ contains a 2-connected bipartite subgraph $P_1\cup C(G) \cong K_{2,3}$. Noticing that there are two vertices which are not isolated outside of $P_1\cup C(G)$, it follows that $pc(G)=2$ from Lemmas \ref{lem4} and \ref{lem6}.

If $|C(G)|=5$, then there exists a Hamilton path or a path of length 6. Since $\delta(G) \geq 2$, it follows that $pc(G)=2$ from Lemma \ref{lem3}.

If $|C(G)|\geq 6$, then by the minimum degree condition, we can find a Hamilton path in $G$. Hence, $pc(G)=2$.

For $n=8$, similarly to the case $n=7$, we also divide the proof into two cases according to the value of $\kappa(G)$.

\begin{figure}
  \centering
 \scalebox{1}{\includegraphics{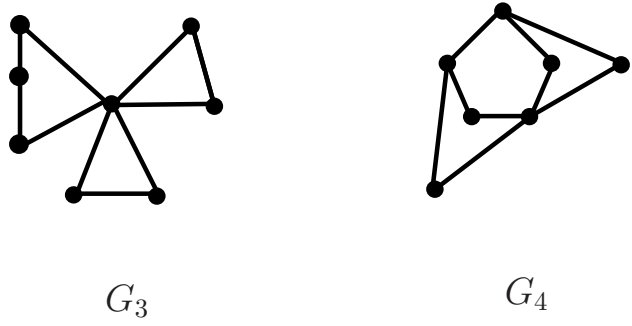}}\\
  \caption{Two subgraphs of $G$ for the case $n=8$.}
\end{figure}

\textbf{Case 3:} $\kappa(G)=1$.

Let $v$ be a cut-vertex of $G$ and let $C_1, \ldots, C_{\ell}$ be the components of $G \backslash v$ such that $|C_1| \leq \ldots \leq |C_{\ell}|$. By the minimum degree condition, we see that $|C_1| \geq 2$, and so $\ell =2$ or $3$. If $\ell =3$, then $G$ contains $G_3$ (see Fig. 2) as a spanning subgraph. It is easy to check that $pc(G_3)=2$. From Lemma \ref{lem1}, we have $pc(G)=2$. Now we let $\ell =2$. Since $n=8$, we further divide the proof into two subcases:

\textbf{Subcase 3.1:} $|C_1|=2$, $|C_2|=5$.

Let $V(C_1)=\{u_1, u_2\}$, $V(C_2)=\{w_1, w_2, w_3, w_4, w_5\}$. Since $\delta(G)\geq 2$, we have that $G[\{v\} \cup C_1]$ is a $K_3$. Since $G$ is connected, there is at least one edge from $v$ to $C_2$. Without loss of generality, we assume $vw_1 \in E(G)$. Since $d_G(w_1)\geq 2$, we know that $w_1$ have at least one neighbor in $\{w_2,w_3, w_4, w_5\}$. Assume $w_1w_2\in E(G)$. Then, there is a path $P=u_2u_1vw_1w_2$. We have $pc(P)=2$. If $|E(w_i, P)|\geq 2$ for $i=3,4,5$, then by Lemma \ref{lem3}, $pc(G)=2$. If there exists one vertex of $\{w_3,w_4,w_5\}$ such that $|E(w_i, P)|=0$, without loss of generality, let $|E(w_3, P)|=0$. Then, $w_3w_4, w_3w_5\in E(G)$ since $d_G(w_3)\ge 2$. If $w_2w_4\in E(G)$ or $w_2w_5\in E(G)$, then there is a Hamilton path $u_2u_1vw_1w_2w_4w_3w_5$ or $u_2u_1vw_1w_2w_5w_3w_4$. Then, $pc(G)=2$. Hence, we assume that $w_2w_i\notin E(G)$  for $i=3,4,5$. Since $d_G(w_2)\ge 2$, we have $w_2v\in E(G)$. Since $C_2$ is connected, we have that there is at least one edge $w_1w_4$ or $w_1w_5 \in E(G)$. We also can find a Hamilton path $u_2u_1vw_2w_1w_4w_3w_5$ or  $u_2u_1vw_2w_1w_5w_3w_4$. Thus, $pc(G)=2$.

Now we may assume that $|E(w_3, P)|=1$. Since $d_G(w_3)\ge 2$, we assume $w_3w_4\in E(G)$. If $w_2w_3\in E(G)$ or $w_2w_4\in E(G)$, then we have have $pc(G)=2$ by Lemma \ref{lem3} since $d_G(w_5)\ge 2$. Hence, $w_2w_3, w_2w_4\notin E(G)$. Then, we have $w_2v\in E(G)$ or $w_2w_5\in E(G)$ since $d_G(w_2)\ge 2$. If $w_2v\in E(G)$, then there is another path $P'=u_2u_1vw_2w_1$ such that $pc(P')=2$. If $w_1w_3\in E(G)$ or $w_1w_4\in E(G)$, then we have $pc(G)=2$ by Lemma \ref{lem3} since $d_G(w_5)\ge 2$. Thus, we have $w_3v,w_4v\in E(G)$ since $E(w_i,P)\ge 1$ for $i=3,4,5$. It is easy to check that $G$ contains either a Hamilton path or $G_3$ as a subgraph since $d_G(w_5)\ge 2$. Thus, $pc(G)=2$. Now we assume that $w_2v\notin E(G),w_2w_5\in E(G)$. Then, there is a path $P''=u_2u_1vw_1w_2w_5$ such that $pc(P'')=2$. If $w_5w_3\in E(G)$ or $w_5w_4\in E(G)$, then we can find a Hamilton path $u_2u_1vw_1w_2w_5w_3w_4$ or $u_2u_1vw_2w_1w_5w_4w_3$. Thus, $pc(G)=2$. Hence, $w_1w_5\in E(G)$ or $w_5v\in E(G)$ since $d_G(w_5)\ge 2$. If $w_5v\in E(G)$, then $G$ contains a Hamilton path since $C_2$ is connected. Thus, $pc(G)=2$. Now we assume that $w_5v\notin E(G)$ and $w_1w_5\in E(G)$. Since $G\neq G_2$, we have $G$ contains a Hamilton path. Thus, $pc(G)=2$.

\textbf{Subcase 3.2:} $|C_1|=3$, $|C_2|=4$.

With similar argument as in the Subcase 1.2 for $n=7$, we can get $pc(G)=2$.

\textbf{Case 4:} $\kappa(G)=2$.

By Fact 1, we have $|C(G)|\geq 4$.

If $|C(G)|=4$, then $C(G)$ has an ear since $G$ is 2-connected. Let $P_1$ be a longest ear of $C(G)$. It is easy to check that $P_1$ has length 2 and the end-vertices of $P_1$ are the antipodal vertices of $C(G)$.  Let $P_2$ be a longest ear of $C(G)\cup P_1$. It is easy to see that $P_2$ also has length 2 and the end-vertices of $P_2$ are the end-vertices of $P_1$. Then, $G$ contains a 2-connected bipartite subgraph $P_1\cup P_2\cup C(G) \cong K_{2,4}$. Noticing that there are two vertices which are not isolated outside of $P_1\cup P_2\cup C(G)$, it follows that $pc(G)=2$ from Lemmas \ref{lem4} and \ref{lem6}.

If $|C(G)|=5$, then it is easy to check that $G$ contains $G_4$ (see Fig. 2) as a subgraph. Thus, $pc(G)=2$ from Lemma \ref{lem3}.

If $|C(G)|\geq 6$, then there exists a Hamilton path or a path of length 7. Since $\delta(G) \geq 2$, it follows that $pc(G)=2$ from Lemma \ref{lem3}.

The proof of Theorem \ref{thm7} is thus complete.
\end{proof}

Combining Theorem~\ref{thm6} and Theorem~\ref{thm7}, we obtain our main result Theorem~\ref{th1}.

\section{A byproduct result}

From the proof of Theorem \ref{thm6}, it seems that we one can get that if $G$ is a bipartite
graph with minimum degree $\delta(G)\geq n/8$ then $pc(G)=2$. But, it is not so. A counterexample of this is
the following constructed graph: Let $G_i$ be a complete bipartite graph
such that each part has $n/8$ vertices for $i = 1, 2, 3, 4$, and take a vertex $v_i \in G_i$ for each $1 \le i \le 4$. Let $G$ be a graph obtained from $G_1, G_2, G_3, G_4$
by joining $v_1$ and $v_j$ with an edge for each $2 \le j \le 4$. Then the resulting graph $G$ is connected and has $\delta(G) = n/8$, but $pc(G)=3$. However, if we increase the minimum degree of $G$ a little bit larger than $n/8$, then we can get $pc(G)=2$.

\begin{thm}\label{thm8}
Let $G$ be a connected bipartite graph of order $n\geq 4$. If $\delta(G)\geq \frac{n+6}{8}$, then $pc(G)=2$.
\end{thm}

\begin{proof}
For $n=4$, we can easily get that $G \cong K_{2,2}$. Thus $pc(G)=2$.

For $5 \le n \le 8$, we can get that $\delta(G) \geq 2$. Then Theorem \ref{thm7} implies that $pc(G)=2$.

For $n \ge 9$, we know that $\delta(G) \ge 2$. If $G$ is 2-edge-connected, then $pc(G)=2$ by Lemmas \ref{lem1} and \ref{lem8}. Next, we assume that $G$ has cut-edges. Let $G^*$ be the bridge-block tree of $G$. From Lemmas \ref{lem6}, \ref{lem8} and \ref{lem9}, we see that the result holds if $\Delta(G^*)\leq 2$. Hence, we only need to consider $\Delta(G^*)\geq 3$ in the following. In order to guarantee the minimum degree of $G$, we know that each end-block of $G$ contains at least $\frac{n+6}{4}$ vertices. We can also get that $G$ has at most three end-blocks. This implies that $\Delta(G^*)\leq 3$, and there is only one vertex attaining the maximum degree of $G^*$. Hence, we have $\Delta(G^*)= 3$.

Note that $\delta(G)\geq 2$. We have that each end-block of $G$ is a maximal connected bridgeless bipartite subgraph, and so it contains at least $4$ vertices. Thus, $n \ge 3 \times 4 +1=13$, which implies that $\delta(G)\geq 3$. Then, each end-block of $G$ contains at least $6$ vertices. Hence, $|V(L)|\geq \max\{6, \frac{n+6}{4}\}$ for each end-block $L$ of $G$.

Let $b_0\in V(G^*)$ be the vertex with degree $3$ and $B_0$ be the subgraph of $G$ corresponding to $b_{0}$. Let $L_i$ be the end-block of $G$ for $i=1,2,3$. Now we claim that $G^*\cong K_{1,3}$. Suppose $G^*\not\cong K_{1,3}$. Then, there is a vertex $v$ with degree $2$ in $G^*$. Denote by $B(v)$ the block of $G$ corresponding to $v$. Therefore, each vertex of $G^*$ other than $b_0$ corresponds to a maximal connected bridgeless subgraph of $G$ since $\delta(G)\geq 3$. By employing the minimum degree condition of $G$, we have that $B(v)$ contains at least $\frac{n+6}{4}$ vertices. Hence, $|V(G)|\geq |B_0|+|L_1|+|L_2|+|L_3|+|B(v)|\geq n+7$, a contradiction. Thus, we can get that $G^*\cong K_{1,3}$.

If $B_0$ is not a singleton, then $|B_0| \ge \frac{n+2}{4}$ since $\delta(G) \ge \frac{n+6}{8}$. Since $|L_i| \ge \frac{n+6}{4}$, we have that $|V(G)| \ge |B_0|+|L_1|+|L_2|+|L_3|\ge \frac{n+2}{4}+3\times \frac{n+6}{4} >n$, a contradiction. Now $B_0$ is a singleton. Since $d_G(B_0)=3$, we have $n \le 16$. On the other hand, we can get that $n\ge |B_0|+|L_1|+|L_2|+|L_3|\ge 1+3\times 6=19$ since $|L_i|\ge 6$ for $i=1,2,3$, a contradiction.

The proof of Theorem~\ref{thm8} is now complete.
\end{proof}

\end{document}